\documentclass[12pt,leqno,a4paper]{amsart}
\usepackage{amssymb,enumerate}
\usepackage{hyperref}                         
\usepackage[all]{xy}

\hypersetup{pdftex,                           
bookmarks=true,
pdffitwindow=true,
colorlinks=true,
citecolor=black,
filecolor=black,
linkcolor=black,
urlcolor=blue,
hypertexnames=true}

\newcommand{\IC}{{\mathbb{C}}}
\newcommand{\IF}{{\mathbb{F}}}
\newcommand{\IZ}{{\mathbb{Z}}}

\newcommand{\fA}{{\mathfrak{A}}}
\newcommand{\fS}{{\mathfrak{S}}}

\newcommand{\cD}{{\mathcal{D}}}
\newcommand{\cP}{{\mathcal{P}}}

\newcommand{\cO}{{\mathcal{O}}}

\newcommand{\Irr}{{\operatorname{Irr}}}

\DeclareMathOperator{\End}{End}

\newcommand{\SL}{{\operatorname{SL}}}

\DeclareMathOperator{\Ind}{Ind}    
\DeclareMathOperator{\Hom}{Hom}    
\DeclareMathOperator{\Syl}{Syl}    
\DeclareMathOperator{\Res}{Res}    
\DeclareMathOperator{\Inf}{Inf}    
\DeclareMathOperator{\proj}{(proj)}    

\newcommand{\tw}[1]{{}^#1\!}

\let\la=\lambda
\let\lra=\longrightarrow

\newcommand{\qbox}[1]{\quad\hbox{#1}\quad}

\newtheorem{thm}{Theorem}[section]
\newtheorem{lem}[thm]{Lemma}

\newtheorem{prop}[thm]{Proposition}

\newtheorem*{thmA}{Theorem A}
\newtheorem*{thmB}{Theorem B}

\theoremstyle{definition}

\theoremstyle{remark}

\begin{document}

\title{Endotrivial modules for the Schur covers of the symmetric and alternating groups}

\date{\today}

\author{Caroline Lassueur}
\address{Caroline Lassueur, FB Mathematik, TU Kaisers\-lautern, Post\-fach
  3049, 67653 Kaisers\-lautern, Germany} 
\email{lassueur@mathematik.uni-kl.de}  
\author{Nadia Mazza}
\address{Nadia Mazza, Department of Mathematics and Statistics, Lancaster
  University, Lancaster, LA1 4YF, UK}
\email{n.mazza@lancaster.ac.uk}    


\thanks{The first author gratefully acknowledges partial financial support by SNF 
Fellowship for Prospective Researchers PBELP2$_{-}$143516.}

\begin{abstract}
We investigate the endotrivial modules for the Schur covers
of the symmetric and alternating groups and determine the structure of their group of endotrivial
modules in all characteristics. We provide a full description of this group by generators and relations in all cases.
\end{abstract}

\maketitle

\pagestyle{myheadings}
\markboth{Caroline Lassueur and Nadia Mazza}{Endotrivial modules for the Schur covers of the symmetric and alternating groups}

\bigskip 

\section{Introduction} \label{sec:intro}
The endotrivial modules over the group algebra $kG$ of a finite group
$G$ of order divisible by the characteristic $p>0$ of a field $k$  have
seen considerable interest over the last forty years, with a surge in
the last fifteen years. 
They have been classified when $G$ is a $p$-group \cite{CT} and many
contributions towards a general classification have been obtained since
for some general classes of finite groups, see e.g. \cite{CMN,CMN1,CHM,
  MaTh07, CMTpsol, CMTquat, CMTtf, LMS, CMNgl, LaMaz14} and the
references therein.  The study of endotrivial $kG$-modules and their 
group of endotrivial modules $T(G)$ (which is finitely generated) is of particular
interest in modular representation theory as it forms an important part
of the Picard group of self-equivalences of the stable category of
finitely generated $kG$-modules. In particular the self-equivalences of
Morita type are induced by tensoring with endotrivial modules. 
 \par 

At present, notable efforts are made to determine the group $T(G)$ 
for quasi-simple groups $G$, as it is hoped that the general question 
of classifying endotrivial modules for arbitrary finite groups can be reduced to
this class of groups (cf.~\cite{CMTpsol}).
The results in this article complement those obtained in the
aforementioned papers. More precisely, we build on the results of \cite{CMN,CHM}, where the group of endotrivial modules  for the
symmetric and alternating groups is fully determined,  to describe $T(G)$ for
their Schur covers.

For $n\geq 4$ the Schur multiplier of the symmetric and alternating
groups is nontrivial. In this case we let $2^{\pm}.\fS_{n}$ denote the
two non-isomorphic double covers of the symmetric group $\fS_{n}$ and
$d.\fA_{n}$ the $d$-fold cover of the alternating group $\fA_{n}$ as in
\cite[\S6.7]{ATLAS}. So $d=2$ or possibly $d\in\{3,6\}$
if $n\in\{6,7\}$. 

Our two main results are as follows.
\bigskip

\pagebreak
\begin{thmA}\label{thm:A}
Let $k$ be an algebraically closed field of characteristic~$2$. The following hold: \smallskip

\begin{center}
\begin{tabular}{llll}
  \rm(1)  &&&   $T(2.\fA_{n})\cong  \begin{cases}
   \IZ/2\oplus\IZ/4 & \text{if }4\leq n\leq 7, \\
   \IZ & \text{if } 8\leq n.
\end{cases}$\\
&&& \\
 \rm(2)   &&&  $T(2^{+}.\fS_{n})\cong  \begin{cases}
\IZ/2\oplus \IZ  & \text{if } n\in\{4,5\} , \\
\IZ^{2}& \text{if } n\in\{8,9\},\\
\IZ & \text{if } n\in\{6,7\}\text{ or } 10\leq n\,.
\end{cases}$ \\
&&& \\
\rm(3) &&&  $T(2^{-}.\fS_{n})\cong  \begin{cases}
\IZ/2\oplus\IZ/4   & \text{if } n\in\{4,5\}, \\
 \IZ& \text{if }6\leq n. 
\end{cases}$
\end{tabular} 
\end{center}
\end{thmA}

\begin{thmB}\label{thm:B}
Let $k$ be an algebraically closed field of characteristic~$p\geq 3$. 
Let $G$ be one of the double covers $2.\fA_{n}$ or $2^{\pm}.\fS_{n}$
with $n\geq 4$. The following hold:
\begin{enumerate}[\hspace{6mm}\rm(1)]
  \item \label{thm:B1}  If $p=3$ and $n=6$, then $T(2.\fA_{6})\cong \IZ/8\oplus
    \IZ^{2}$ and
    $T(2.\fS_{6})=\Inf_{\fS_{6}}^{2.\fS_6}(T(\fS_{6}))\cong(\IZ/2)^2\oplus\IZ$.
  \item \label{thm:B2}  If $n\geq p+4$, then the inflation homomorphism
    $\Inf_{G/Z(G)}^G : T(G/Z(G))\lra T(G)$ is an isomorphism. 
  \item \label{thm:B3}   If $p\geq3$ and $p\leq n \leq min\{2p-1, p+3\}$  then $T(G)\cong \IZ/2e\oplus\IZ/2$,
    where $e$ denotes the inertial index of the principal block of
    $kG$. 
\end{enumerate}  
\end{thmB}

In addition, when the torsion-free rank of $T(G)$ is greater than $1$,
we provide generators for the torsion-free part in all cases, and
similarly when the torsion subgroup is nontrivial, then we describe its
elements explicitly. 
\par
The paper is organised as follows. In Section~\ref{sec:pre}, we recall the 
necessary background on endotrivial modules and in Section~\ref{sec:schur} that about the Schur
covers of alternating and symmetric groups.
 Theorem~A is proved in Section~\ref{sec:char2} and Theorem~B in
Section~\ref{sec:charodd}. In both these sections we describe the groups
$T(G)$ by generators and relations.
Finally, in Section~\ref{sec:exc}, we handle the structure of the group of endotrivial modules of the four
exceptional Schur covers $3.\fA_{6}$, $6.\fA_{6}$, $3.\fA_{7}$ and
$6.\fA_{7}$.

\bigskip

\noindent \textbf{Acknowledgements.} The authors wish to thank J\"urgen
M\"uller for his help with the proof of Proposition~\ref{prop:2rank4}, and Gunter Malle for useful comments on a preliminary version of this paper.

\bigskip
\section{Preliminaries} \label{sec:pre}

Throughout, unless otherwise stated, we let $G$ denote a finite group, 
$p$ a prime number dividing the order of $G$,
$k$ an algebraically closed field of characteristic $p$, and $(K,\cO,k)$
a splitting $p$-modular system for $G$ and its subgroups. All
$kG$-modules we consider are assumed to be finitely generated left
$kG$-modules. We denote by $B_{0}(kG)$ the principal block of $kG$ and
by $\Irr(G)$ the set of irreducible complex characters of $G$. We refer
the reader to the standard literature for these concepts (e.g. \cite{BENSON1}).  

\subsection{Background results on endotrivial modules}\label{ssec:et}  \

A $kG$-module $M$ is called { \emph{endotrivial}} if its $k$-endomorphism
ring decomposes as a $kG$-module into a direct sum
$\End_{k}(M)\cong M^{*}\otimes_{k} M\cong k\oplus \text{(proj)}$, where
$M^{*}=\Hom_k(M,k)$ denotes the $k$-dual of $M$, $k$ the trivial
$kG$-module and $\text{(proj)}$ some projective module (possibly zero). 
If $M$ is endotrivial, then $M$ splits as $M\cong
M_{0}\oplus\text{(proj)}$, where $M_{0}$ is, up to isomorphism, the
unique indecomposable endotrivial direct summand of $M$. The
set $T(G)$ of isomorphism classes of indecomposable endotrivial
  $kG$-modules is an abelian group for the binary operation induced by the
tensor product $\otimes_k$ of endotrivial modules, that is
$[M]+[N]=[(M\otimes_{k} N)_{0}]$. The identity element is the class
of the trivial module $[k]$, and we have $-[M]=[M^{*}]$ for any $[M]\in
T(G)$.  
We denote by $X(G)$ the group of isomorphism classes of
one-dimensional $kG$-modules (endowed with $\otimes_{k}$). 
Recall that $X(H)\cong (G/[G,G])_{p'}$, the
$p'$-part of the abelianisation of $G$. In particular $X(G)$ is
isomorphic to a finite subgroup of $T(G)$. \par
By \cite[Corollary 2.5]{CMN1}, the group $T(G)$ is known to be finitely
generated (note that this is a consequence to the corresponding result originally obtained by Puig when $G$ is a $p$-group) 
so that we may write $T(G)=TT(G)\oplus TF(G)$, where $TT(G)$
is the torsion subgroup of $T(G)$ and $TF(G)$ a torsion-free
complement. The $\IZ$-rank of $TF(G)$ is called the \emph{torsion-free
  rank} of $T(G)$. Moreover, if the torsion-free rank is one, then we may choose $TF(G)$ such that 
$TF(G)=\left<[\Omega_{G}(k)]\right>$, where $\Omega_{G}$ denotes the Heller
translate. We write $T_0(G)$ for the subgroup of $T(G)$ formed by the
isomorphism classes of the indecomposable endotrivial modules which
belong to the principal block $B_0(kG)$. \\

We start with the following general lifting result which allows us to
use ordinary character theory in our investigations. Hereafter the term
{\em $\cO G$-lattice} means an $\cO$-free $\cO G$-module.

\begin{thm}[{}{\cite[Theorem 1.3 and Cor. 2.3]{LMS}}]\label{thm:lift}\vbox{\ }
\noindent Let $(K,\cO,k)$ be a splitting $p$-modular system for a finite group
$G$. Let $M$ be a $kG$-module. 
\begin{enumerate}[\hspace{6mm}\rm(1)]
\item \label{thm:lift1}  If $M$ is endotrivial, then $M$ is liftable to an endotrivial
  $\cO G$-lattice. 
\item \label{thm:lift2} Suppose that $M$ is endotrivial and lifts to a $\IC G$-module
  affording the character $\chi$. Then $|\chi(g)|=1$  for every
  $p$-singular element $g\in G$.   
\end{enumerate}
\end{thm}

The class of endotrivial modules is stable under the operations of
restriction and inflation, under the following assumptions.

\begin{lem}\label{lem:resinf}
Let $G$ be a finite group and $H$ a subgroup of $G$.
\begin{enumerate}[\hspace{6mm}\rm(1)]
  \item\label{lem:resinf1} If $p\,|\,|H|$, then restriction along the
inclusion $H\hookrightarrow G$ induces a group homomorphism 
$\Res^{G}_{H}:T(G)\longrightarrow T(H)$ by mapping
    $[M]\mapsto[(M\!\!\downarrow^{G}_{H})_{0}]$.  If, moreover, $H$ contains the normaliser of a Sylow
  $p$-subgroup in $G$, then $\Res^{G}_{H}$ is injective, with partial
  inverse induced by Green correspondence. 
  \item \label{lem:resinf2} If $Z\trianglelefteq G$ is a $p'$-subgroup of $G$, then restriction along the quotient map $G\to G/Z$
   induces an injective group homomorphism
$\Inf_{G/Z}^{G}:T(G/Z)\longrightarrow T(G)$ by mapping
    $[M]\mapsto[\Inf_{G/Z}^{G}(M)]$, where $\Inf_{G/Z}^{G}(M)$ denotes the module $M$ regarded as a $kG$-module with trivial action of $Z$.
\end{enumerate}  
\end{lem}

For part (\ref{lem:resinf1}) see \cite[Prop. 2.6 and Rem. 2.9]{CMN1}, and for
part (\ref{lem:resinf2}) see \cite[Lemma~3.2(1)]{LaMaz14}. 

\smallskip
We now present the results needed in the sequel in order to investigate the
structure of $T(G)$ as a finitely generated abelian group. We start with
the torsion-free rank of $T(G)$. Recall that the {\em $p$-rank} of a
finite group $G$ is the logarithm to base $p$ of the maximum of the
orders of the elementary abelian $p$-subgroups of $G$. Also, by {\em
  maximal elementary abelian $p$-subgroup} of $G$, we mean an elementary
abelian $p$-subgroup which is not properly contained in another
elementary abelian $p$-subgroup of $G$.

\begin{thm}[{}{\cite[Theorem 3.1]{CMN1}}]\label{thm:ranktf}
The torsion-free rank of $T(G)$ is equal to the number of
$G$-conjugacy classes of maximal elementary abelian $p$-subgroups of rank
$2$ if $G$ has $p$-rank at most $2$, or that number plus one if $G$ has
$p$-rank greater than $2$. 
\end{thm}
In contrast $T(G)=TT(G)$ is a finite for groups with Sylow
$p$-subgroups of $p$-rank $1$ (i.e. cyclic or possibly generalised
quaternion if $p=2$).

Moreover if $TF(G)$ is not cyclic, then we do not have means to find generators for it, 
except in some very special instances, as in the case of central extensions by nontrivial
$p$-subgroups for which Theorem~\ref{thm52:CMT} below applies.  
Independently from considerations about generators, if we only regard $TF(G)$ as a 
torsion-free abelian group, there are two bounds which are relevant to
the study of endotrivial modules.

\begin{thm}\label{thm:tf-rank}
Let $G$ be a finite group.
\begin{enumerate}[\hspace{6mm}\rm(1)]
\item\label{tf-it1} If $G$ has a maximal elementary abelian $p$-subgroup
  of rank $2$, then $G$ has $p$-rank at most $p$ if $p$ is odd, or at
  most $4$ if $p=2$. Moreover both bounds are optimal. 
\item\label{tf-it2} $G$ has at most $p+1$ $G$-conjugacy classes of maximal
  elementary abelian $p$-subgroups of rank $2$ if $p$ is odd, or at most
  $5$ if $p=2$. Moreover both bounds are optimal.
\end{enumerate}
Consequently if $G$ has $p$-rank greater than $p$ if $p$ is odd, or
greater than $4$ if $p=2$, then $TF(G)=\left<[\Omega_{G}(k)]\right>\cong \IZ$
is infinite cyclic.
\end{thm}

Part~(\ref{tf-it1}) was proved by MacWilliams \cite[Four
Generator Theorem, p. 349]{MacW} for $p=2$,
and by Glauberman and the second author \cite{GM} for $p$
odd. Part~(\ref{tf-it2}) was proved by Carlson \cite{Caelab} for
$p=2$ and by the second author \cite{MAZposet} for $p$ odd.

The groups $G$ we handle in this paper are central extensions by cyclic
subgroups of order dividing $6$. Therefore, in investigating $T(G)$ in
characteristic $2$ and $3$, there are cases in which the following
theorem enables us to obtain generators for $TF(G)$. 
The key result on the $p$-local structure of a finite
$p$-group which is used in Theorem~\ref{thm52:CMT} is the
following (cf. \cite[Section~3]{CMTtf}). 
Suppose that $P$ is a non-abelian finite $p$-group with
$TF(P)$ not cyclic. Thus there is at least one maximal elementary
abelian subgroup, say $E=\langle z,s\rangle\cong C_p\times C_p$, with
$z\in Z(P)$ and $s\notin Z(P)$. 
It turns out that $C_P(E)=C_P(s)=\langle s\rangle\times L$ where $L$ has
$p$-rank $1$ and its isomorphism type is independent of the choice of a
maximal elementary abelian subgroup $E$. Hence we can define a constant
$m$ which only depends on $P$ as follows:
$$m=\left\{\begin{array}{ll}
1&\qbox{if $|L|\leq2$;}\\
2&\qbox{if $L$ cyclic of order $\geq3$;}\\
4&\qbox{if $L$ is generalised quaternion ($p=2$).}
\end{array}\right.$$

\begin{thm}[{}{\cite[Theorem~5.2]{CMTtf}}]\label{thm52:CMT}
Let $G$ be a group such that the torsion-free rank~$n_G$ of~$T(G)$ 
is at least~$2$. Let $P\in\Syl_{p}(G)$, let $Z$ be 
the unique central subgroup of~$P$ of order~$p$, and assume 
that $Z$ is normal in~$G$. For $2\leq i\leq n_G$ choose a representative
$E_i$ of a conjugacy class of maximal elementary abelian subgroups of
$P$ of rank $2$.
Let $a=mp$ where $m$ is the integer defined
above. 
\begin{enumerate}[\hspace{6mm}\rm(1)]
\item For $2\leq i\leq n_G$, there is a subquotient $N_i$ of 
the $kG$-module $\Omega^a_{G}(k)$ which is endotrivial, and such that
$$N_{i}\!\downarrow^{\!G}_{\!E_i}\cong
\Omega_{E_{i}}^a(k)\oplus\proj\qbox{and}
N_{i}\!\downarrow^{\!G}_{\!E_j}\cong
k\oplus\proj\qbox{whenever $j\neq i$.}$$
\item $TF(G)$ is generated by $[\Omega_{G}(k)]$, 
$[N_2]$, \ldots, $[N_{n_G}]$.
\end{enumerate}
\end{thm}

\medskip
Next we describe the known results on  the torsion subgroup $TT(G)$ of
$T(G)$. Given $P\in\Syl_p(G)$, we define
$$K(G)=\ker\left(\Res^{G}_{P}:T(G)\lra T(P)\right).$$
In other words, $K(G)$ is the subgroup of $T(G)$ formed by the equivalence
classes of the trivial source endotrivial modules. In particular,
$X(G)\leq K(G)\leq TT(G)$.

\begin{lem}\label{lem:K(G)} 
Let $G$ be a finite group and $P\in\Syl_{p}(G)$.
\begin{enumerate}[\hspace{6mm}\rm(1)]
\item\label{lem:K(G)1} $TT(P)=\{[k]\}$ unless $P$ is cyclic, generalised quaternion,
  or semi-dihedral. 
\item\label{lem:K(G)2} $K(G)=TT(G)$ whenever $TT(P)=\{[k]\}$.
\item\label{lem:K(G)3} If $\tw{x}{P}\cap P$ is nontrivial for all $x\in G$, then $K(G)=X(G)$.
\item \label{lem:K(G)4} If $1\lra Z\lra G\lra H\lra 1\,\,$ is a central
  extension with $Z=Z(G)$ of order divisible by~$p$, then $K(G)=X(G)$. 
\end{enumerate}  
\end{lem}
Part (\ref{lem:K(G)1}) is a consequence of the classification of endotrivial modules over
$p$-groups (see \cite{CT}), Part~(\ref{lem:K(G)2}) is \cite[Lemma 2.3]{CMTpsol},
part~(\ref{lem:K(G)3}) is \cite[Lemma 2.6]{MaTh07}, while part (\ref{lem:K(G)4}) is a particular
case of~(\ref{lem:K(G)3}). 
For the structure of $TT(G)$ when $P\in\Syl_p(G)$ is generalised
quaternion or semi-dihedral, we refer the reader directly to
\cite{CMTquat}. 
We review here known facts about $T(G)$ when $P\in\Syl_p(G)$ is cyclic.

\begin{thm}[{}{\cite[Theorem 3.2]{MaTh07} and
      \cite[Lemma 3.2]{LMS}}]\label{thm:cyc} 
Let $G$ be a finite group with a cyclic Sylow $p$-subgroup $P$ of order
at least $3$. Let $Z$ be the unique subgroup of $P$ of order $p$ and let
$H=N_G(Z)$. Let $e=|N_{G}(Z):C_{G}(Z)|$ denote the inertial index of the
principal block $B_0(kG)$ of $kG$. The following hold.
\begin{enumerate}[\hspace{6mm}\rm(1)]
 \item \label{thm:cyc1}  $T(G)=\{\; [\Ind_H^G(M)] \,\mid\, [M]\in T(H) \;\}\cong T(H)\,.$
\item\label{thm:cyc2}  The sequence
$$\xymatrix{0\ar[r]& X(H)\ar[r]& T(H)\ar[r]^{\Res^H_P}&T(P)\ar[r]&
  0}$$
is exact. Moreover $T(P)=\langle[\Omega_{P}(k)]\rangle\cong \IZ/2$ and the sequence
splits if $e$ is odd.
\item\label{thm:cyc3}  $T_{0}(G)=\left<[\Omega_{G}(k)]\right>\cong \IZ/2e$.
\item\label{thm:cyc4}  The number of $p$-blocks of $kG$ containing indecomposable 
  endotrivial $kG$-modules is equal to $\frac{|X(H)|}{e}$, and each
  of these has inertial index equal to $e$ and contains a simple
  endotrivial $kG$-module. 
\end{enumerate}  
\end{thm}

An obvious consequence of Theorem~\ref{thm:cyc}(\ref{thm:cyc4})  is that
in the case of cyclic Sylow $p$-subgroups, we have
$T(G)=\left<[\Omega_{G}(k)], [S_{2}],\ldots, [S_{|X(H)|/e}]\right>$, where
the modules $S_{i}$ for \linebreak $2\leq i\leq |X(H)|/e$ are simple endotrivial
modules in pairwise distinct non-principal blocks of $kG$.


\section{The Schur covers of the alternating and symmetric
  groups}\label{sec:schur}\label{ssec:2rk} 

For a detailed construction of the Schur covers of the alternating and
symmetric groups, we refer the reader to \cite[Chap. 2]{HH} and
\cite[\S2.7.2]{Wilson}. 
We recall that for $n\geq 4$ the Schur multiplier of the symmetric and
alternating groups is nontrivial of order $2$, except in the cases of
$\fA_{6}$ and $\fA_{7}$ in which case it has order~$6$.
We let $2^{\pm}.\fS_{n}$ denote the two isoclinic double covers of
the symmetric group $\fS_{n}$, and $d.\fA_{n}$ denote the $d$-fold Schur
cover of the alternating group $\fA_{n}$.
As in \cite[\S~6.7]{ATLAS}, the group $2^+.\fS_n$ is the double cover of
$\fS_n$ in which transpositions of $\fS_n$ lift to involutions, while in
$2^-.\fS_n$ transpositions lift to elements of order $4$. We recall that
$2^{+}.\fS_{n}\cong 2^{-}.\fS_{n}$ if and only if $n=6$. 

We use the presentation by generators and relations of
$2^\pm.\fS_n$ given in \cite{HH,ATLAS}, and  the convention
$[g,h]=g^{-1}h^{-1}gh$ for commutators. So
$$2^\pm.\fS_n=\langle z,t_1,\dots,t_{n-1}\rangle\big/
\langle\mathcal R\rangle$$ 
where $\mathcal R$ is generated by the relations
\begin{align*}
z^2&=1;\\
t_j^2&=z^\alpha\qbox,1\leq j\leq n-1;\\
(t_jt_{j+1})^3&=z^\alpha\qbox,1\leq j\leq n-2;\\
[z,t_j]&=1\qbox,1\leq j\leq n-1;\\
[t_j,t_k]&=z\qbox{if } |j-k|>1,1\leq j,k\leq n-1;
\end{align*}
where $\alpha=0$ for $G=2^+.\fS_n$ and $\alpha=1$ for
$G=2^-.\fS_n$. Then $2.\fA_n$ is the preimage of $\fA_n$ in either
double cover of $\fS_n$ via the natural quotient map $G\mapsto G/Z(G)$,
where $Z(G)=\langle z\rangle$. It follows that 
$$2.\fA_n=\langle z,x_1,\dots,x_{n-2}\rangle\big/
\big(\langle\mathcal R\rangle\cap\langle z,x_1,\dots,x_{n-2}\rangle\big)$$
where $x_j=t_jt_{j+1}$ for $1\leq j\leq n-2$. \par
For later use, we record a few combinatorial equalities from
\cite[\S~2.7.2]{Wilson}. Write $t_j=[j,j+1]$ for an element of
$2^\pm.\fS_n$ which lifts the transposition $(j,j+1)$. Then, the lifts of
transpositions are of the form $\pm[i,j]$ and have order $2$ in
$2^+.\fS_n$, respectively $4$ in $2^-.\fS_n$. In particular we calculate  
$[1,2]^{[1,2]}=-[2,1]$ and
$[1,2]^{[3,4]}=-[1,2]$, so that $([1,2][3,4])^2=-1$ and so,
independently of the order of the lift of a transposition, we conclude
that in $2.\fA_n$ the noncentral involutions are the lifts of $l$-fold transpositions for
$l\equiv0\pmod4$. For convenience, we call a permutation $g$ an {\em
  $l$-fold transposition}, for some positive integer $l$, if $g$ is the
product of $l$ disjoint transpositions. In particular transpositions are
$1$-fold transpositions and $(1,2)(3,4)$ is a $2$-fold transposition.
The noncentral involutions of $2^-.\fS_n$ are
the lifts of $l$-fold transpositions with $l\equiv3\pmod4$, whereas the
latter lift to elements of order $4$ in $2^+.\fS_n$. \\

Next we recall some facts about the ordinary character theory of $2^{\pm}.\fS_{n}$ (see  \cite[\S6 and \S8]{HH}).
Let $\cD(n)$ denote the set of partitions of $n$ into distinct parts, i.e. 
\linebreak $\la=(\la_1>\la_2>\dots>\la_t)$ with $\sum_j\la_j=n$.
A partition $\la$ is called {\em even}, respectively {\em odd}, if its 
number of even parts  is even, respectively odd. 
The faithful complex irreducible characters of $2^{\pm}.\fS_{n}$ 
are parametrised by the partitions $\la\in\cD(n)$ as follows
(see \cite[Theorem~8.6]{HH}): 
\begin{enumerate}[\hspace{6mm}\rm(1)]
\item If $\la$ is even, then there is one
irreducible character $\psi_\la\in\Irr(2^{\pm}.\fS_{n})$ which splits
upon restriction to $2.\fA_n$ into two distinct irreducible constituents
$\psi_\la^\pm$.
\item If $\la$ is odd, then there are two irreducible characters
  $\psi_\la^\pm\in\Irr(2^{\pm}.\fS_{n})$ which have the same irreducible
  restriction to $2.\fA_n$.
\end{enumerate}
Moreover, the necessary information for our investigation on the values
of the faithful elements in $\Irr(2^\pm.\fS_n)$ and $\Irr(2^\pm.\fA_n)$
is provided by \cite[Theorem~8.7]{HH}. 
In particular, for  $g\in 2^\pm.\fS_n$ and $\la\in\cD(n)$, the following hold.
\begin{enumerate}[\hspace{6mm}\rm(1)]
\item If $\la$ is odd, then $\psi_\la(g)=0$ if $g$ lifts a cycle type  of
  $\fS_{n}$ which is not in $\cP^0(n)\cup\{\la\}$, where
  $\cP^0(n)$ denotes the set of partitions of $n$ with only odd parts. 
\item If $\la$ is even, then $\psi_\la(g)=0$ if $g$ lifts a cycle type of
  $\fS_{n}$  which is not in $\cP^0(n)$.
\item If $\la$ is even, then the two irreducible constituents
  $\psi_{\lambda}^{\pm}$ of the restriction of $\psi_{\lambda}$ to
  $2.\fA_{n}$ are such that
  $\psi_{\lambda}^{+}(g)=\psi_{\lambda}^{-}(g)$ for all $g\in
  2.\fA_{n}$ such that $g$ does not project to an element of cycle type $\la\in\fA_{n}$. 
\end{enumerate}

\bigskip

Finally, motivated by Theorem~\ref{thm:ranktf} and Lemma~\ref{lem:K(G)},
we summarise in Table~\ref{tab:2rk} below the relevant information on
the Sylow $2$-subgroups of the Schur covers of the alternating and
symmetric groups and their $2$-rank.

It is known that the $2$-rank of $2.\fA_n$ is equal to 
$3\lfloor\frac n8\rfloor+1$, see e.g. \cite[Prop.~5.2.10]{GLS3}. To work
out the isomorphism types of the Sylow $2$-subgroups, we used
\cite{GLS3,ATLAS,MAGMA}.  
For $n\geq 3$, we denote by $D_{2^n}$ and $Q_{2^{n}}$ a dihedral and a
generalised quaternion group of order $2^{n}$ respectively, and for
$n\geq4$, we write $SD_{2^n}$ for a semi-dihedral group of order $2^n$. 
In addition, $R_{m,r}$ denotes the group {\textsf{SmallGroup}}$(m,r)$,
of order $m$, in the MAGMA Database of Small Groups (see \cite{MAGMA}).

{\footnotesize
\begin{table}[htbp]
 \caption{Schur covers of the symmetric and alternating groups with
   2-rank less than~5 and their Sylow 2-subgroups.} 
  \label{tab:2rk}
$$\begin{array}{|l|c|c||l|c|c|}
\hline
G&\text{2-rank}&P\in \Syl_{2}(G)&G&\text{2-rank}&P\in \Syl_{2}(G)\\
\hline\hline
2.\fA_4, 2.\fA_5&1&Q_8&2^-.\fS_4,2^-.\fS_5&1&Q_{16}\\
\hline
2.\fA_6,~2.\fA_7,~6.\fA_6,~6.\fA_7&1&Q_{16}&2^+.\fS_4,2^+.\fS_5&2&\hbox{SD}_{16}\\
\hline
3.\fA_6,~3.\fA_7&2&D_8&2.\fS_6,~2^\pm.\fS_7&2&R_{32,44}\\
\hline
2.\fA_n, 8\leq n\leq 15&4&& 2^\pm.\fS_n, 8\leq n\leq 15&\geq4&\\
\hline
\end{array}$$
\end{table}   
}

\bigskip
\section{Double covers in characteristic $2$} \label{sec:char2}

Throughout this section $p=2$ and $k$ denotes an algebraically closed
field of characteristic $2$.   We prove Theorem~A.

\vspace{2mm}

\subsection{The structure of $T(G)$ when the $2$-rank is one.} \

From Table~\ref{tab:2rk}, the Schur covers of the alternating and
symmetric groups have $2$-rank one if and only if they have generalised
quaternion Sylow $2$-subgroups.

\begin{prop}\label{prop:2-rank1}
Let $G$ be one of the groups $2.\fA_n$ for $4\leq n\leq 7$, or $2^{-}.\fS_n$ for
$n\in\{4,5\}$. Then $T(G)\cong  \IZ/2\oplus\IZ/4$.
\end{prop}

\begin{proof}
A Sylow $2$-subgroup $P$ of $G$ is generalised
quaternion of order $8$ for $2.\fA_4$ and $2.\fA_5$, and of order $16$
otherwise, see Table~\ref{tab:2rk}. Therefore in all cases $T(P)\cong \IZ/2\oplus\IZ/4$ by
\cite[Theorem 6.3]{CT2000}. Now $Z(G)$ is a nontrivial normal $2$-subgroup
of $G$ and $X(G)\cong (G/[G,G])_{2'}$ is trivial in all cases. By \cite[Theorem
  4.5]{CMTquat} we conclude that $T(G)\cong T(P)$ as required.
\end{proof}

Note that the $\IZ/4$ summand is generated by the class of $\Omega_{G}(k)$,
since a projective resolution of the trivial module for $G$ is periodic
of period $4$. 
The construction of the generator of the $\IZ/2$
summand is detailed in the proof of \cite[Theorem 4.5]{CMTquat}.


\medskip

\subsection{The structure of $T(G)$ when the $2$-rank is at least $2$} \

As pointed out in Section~\ref{sec:pre}, in order to determine the
torsion subgroup $TT(G)$ of $T(G)$ it suffices to find the trivial
source endotrivial modules, i.e. the group $K(G)$. Our objective further
simplifies, because it is well known that Sylow $2$-subgroups
of symmetric and alternating groups are selfnormalising except in few
small groups (cf. \cite[Corollary 2 and Theorem
  p. 124]{W25}). So the same holds in a double cover
of an alternating or symmetric group.

\begin{prop}\label{prop:TTG2rank>2}
 Let $G$ be a double cover of an alternating or symmetric
 group, whose $2$-rank is at least~$2$. Then $TT(G)=\{[k]\}$, unless $G=2^+.\fS_4$ or $2^+.\fS_5$ in which case $TT(G)\cong\IZ/2$.
\end{prop}

\begin{proof}
Let $G$ be a double cover of an alternating or symmetric group of degree
at least $6$. Then a Sylow $2$-subgroup $P$ of $G$ is selfnormalising
and $TT(P)=\{[k]\}$ by Lemma~\ref{lem:K(G)}. 
So $X(N_G(P))=\{[k]\}$ and a fortiori $TT(G)=\{[k]\}$.

Now suppose that $G=2^+.\fS_4$ or $G=2^+.\fS_5$, and let
$P\in \Syl_{2}(G)$. Then $P\cong SD_{16}$ is semi-dihedral of order
$16$ and $P=N_G(P)$. Thus \cite[Prop. 6.5]{CMTquat} and \cite[Theorem 
  7.1]{CT} yield $T(G)\cong \Res^G_P( T(G) )\cong T(P) \cong\IZ/2\oplus\IZ$\,. The claim follows.
\end{proof}

We refer the reader to \cite[Theorem 6.4(2)]{CMTquat} for a more precise
description of the nontrivial selfdual endotrivial $kG$-module of a
group $G$ with semi-dihedral Sylow $2$-subgroups.

\medskip
We now handle the group $TF(G)$.  
From Theorem~\ref{thm:ranktf} and Corollary~\ref{thm:tf-rank}, we gather
that $TF(G)=\langle[\Omega_{G}(k)]\rangle\cong\IZ$ whenever $G$ is a double
  cover of an alternating or symmetric group of degree at least $16$,
  because then the $2$-rank is greater than $4$ (see Table~\ref{tab:2rk}). 
Thus we are left with the cases when the $2$-rank of $G$ is between $2$
and $4$ (see Table~\ref{tab:2rk}), and we want to find which of these, if any, have maximal
Klein-four subgroups.

\begin{prop}\label{prop:2rank4}
Let $G$ be one of the groups $2.\fA_n$ or $2^\pm.\fS_n$ with 
$8\leq n\leq 15$. Then $G$ has no maximal Klein-four subgroups unless
$G=2^+.\fS_n$ with $n\in\{8,9\}$, in which case there is a unique
$G$-conjugacy class of maximal Klein-four subgroups. 
\end{prop}

\begin{proof}
We first consider $G=2.\fA_n$ for $8\leq n\leq15$ and use the
presentation of $G$ given in Section~\ref{sec:schur}.
It suffices to show that the
centraliser of any involution contains at least three commuting involutions.
We use the notation introduced in Section~\ref{sec:schur}, and the
explicit calculations in \cite[\S2.7.2]{Wilson}, briefly recalled
in Section~\ref{sec:schur}. 

The involutions of $G$ are the lifts of
$l$-fold transpositions for $l\equiv0\pmod4$, and any involution of $G$
is $G$-conjugate to the lift $x=[1,2][3,4][5,6][7,8]$ of
$(1,2)(3,4)(5,6)(7,8)$.  
So $y=[1,3][2,4][5,7][6,8]$ centralises $x$ and
$xy=yx=[1,4][2,3][5,8][6,7]$ 
is an involution too. Therefore $C_G(x)\geq\langle x,y,z\rangle\cong
C_2^3$ as required.

Let us now take $G=2^\pm.\fS_n$ for $8\leq n\leq 15$. 
Proceeding as for
the alternating group, we calculate the $2$-ranks of the centralisers of
any involution of $G$. The involutions of $G$ are those of $2.\fA_n$
together with the $l$-fold transpositions with  
$$l\equiv \begin{cases}
1\pmod4&\text{if } G=2^+.\fS_n\,,\text{ or}\\
3\pmod4&\text{if } G=2^-.\fS_n\,,
\end{cases}$$ 
as noted in Section~\ref{sec:schur}.
 
Building on the first part of the proof, it suffices to consider
involutions lifting odd permutations and find their
centralisers. Now $C_{2.\fA_n}(x)$ has index $2$ in $C_G(x)$, which
gives $C_G(x)=\langle x\rangle\times C_{2.\fA_n}(x)$. 
Because $z\in C_{2.\fA_n}(x)$, the $2$-rank of $C_G(x)$ is at least $3$
if and only if there exists a noncentral involution in
$2.\fA_n$ which centralises $x$. If this holds, then $G$ cannot have any
maximal Klein-four subgroup.

First assume that $n\geq10$ and $G=2^-.\fS_n$. Any involution not in
$2.\fA_n$ is conjugate to the lift of the $3$-fold transposition
$x=[1,2][3,4][5,6]$, or for $n\in\{14,15\}$ possibly the $7$-fold transposition
$x'=[1,2][3,4]\cdots[13,14]$. 
In the first case, we take $y=[1,3][2,4][5,6][7,8]$, which gives
$xy=yx=[1,4][2,3][7,8]$. Thus
$C_G(x)\geq\langle x,y,z\rangle\cong C_2^3$ as required. 
For $x'$, we take $y=[1,2][3,4][5,6][7,8]$, so that 
$x'y=yx'=[9,10][11,12][13,14]$ and we are done in this case.

Suppose $n\geq 10$ and $G=2^+.\fS_n$. Any involution not in
$2.\fA_n$ is conjugate to either
$x=[1,2]$, or $x'=[1,2][3,4]\cdots[9,10]$.
Now $xx'=x'x=[3,4][5,6][7,8][9,10]$ so that
$C_G(x)\geq\langle x,x',z\rangle\cong C_2^3$ and we are done for these
groups too. 

Suppose now $G=2^-.\fS_n$ with $n\in\{8,9\}$. By the above, any involution of
$G$ not in $2.\fA_n$ is conjugate to $x=[1,2][3,4][5,6]$. 
Let $y=[1,2][3,5][4,6][7,8]$, so that $xy=yx=[3,6][4,5][7,8]$, and
we get $C_G(x)\geq\langle x,y,z\rangle\cong C_2^3$ as required.

We are left with $2^+.\fS_n$ and $n\in\{8,9\}$. We claim that
$G$ has a maximal Klein-four subgroup, unique up to conjugacy, namely
$\langle x,z\rangle$ where $x=[1,2]$. As noted above, it is
enough to prove that there is no noncentral involution in $2.\fA_n$
which centralises $x$. Any such involution must be the lift of a $4$-fold
transposition. The centraliser in $\fS_n$ of the transposition $(1,2)$ is the direct
product $\langle(1,2)\rangle\times\fS_{n-2}$ where the latter symmetric
group permutes the set $\{3,\dots,n\}$ of cardinality $6$ or $7$, and so
does not contain any $4$-fold transposition whose lift in $2.\fA_n$ is
an involution. Therefore $C_{2.\fA_n}(x)$ has Sylow $2$-subgroup $Q$ of
rank $1$. Indeed, by the results for $2.\fA_6$ and $2.\fA_7$ in
Table~\ref{tab:2rk}, we deduce that $Q$ is quaternion of order $16$.
\end{proof}

As a consequence of Proposition~\ref{prop:2rank4}, $TF(G)\cong\IZ$ in
characteristic $2$ for $G$ a double cover of an alternating or symmetric
group of degree at least $8$, except for $2^+.\fS_8$ and $2^+.\fS_9$, in
which case the torsion-free rank is equal to $2$. Actually, we can be
more thorough and give generators for $TF(G)$ when $G=2^{+}.\fS_n$
with $n\in\{8,9\}$. 

\begin{thm}\label{thm:tfchar2}
Let $G$ be a double cover of an alternating or symmetric group of
$2$-rank at least~$2$. That is, $G$ is one of $2.\fA_{n}$ with $n\geq 8$,
or $2^\pm.\fS_n$ with $n\geq6$, or $2^+.\fS_n$ with $n\in\{4,5\}$. Then:
$$TF(G)\cong \begin{cases}
\IZ^{2} & \text{if } G\text{ is one of }2^+.\fS_8,2^+.\fS_9,\\
 \IZ  & \text{otherwise.}\\
\end{cases}$$
Furthermore, if $G=2^+.\fS_n$ with $n\in\{8,9\}$, let $F$ be a maximal
Klein-four subgroup and let $E$ be any non-maximal Klein-four
subgroup of $G$. Then 
$$\smallskip TF(G)=\langle[\Omega_G(k)],[M]\rangle\cong\IZ^2\,,$$
where 
$$ \Res^G_F(M)\cong\Omega_{F}^8(k)\oplus\proj\qbox{and}
\Res^G_E(M)\cong k\oplus\proj \,.$$
\end{thm}

\begin{proof}
By Theorem~\ref{thm:ranktf}, we need to find the conjugacy classes of
maximal Klein-four subgroups of $G$. By
\cite[Prop.~5.2.10]{GLS3}, if $G$ is one of $2.\fA_{n}$ or
$2^\pm.\fS_n$ with $n\geq16$, then $G$
has $2$-rank greater than $4$, so that Theorem~\ref{thm:tf-rank} shows that
$G$ cannot have maximal Klein-four subgroups.
If $G=2^+.\fS_4$ or $G=2^+.\fS_5$, then a Sylow $2$-subgroup of $G$ is
semi-dihedral (of order $16$), and so has rank $2$ and a unique
conjugacy class of (maximal) Klein-four subgroups. 
If $G=2.\fS_6$ or $G=2^\pm.\fS_7$, then a Sylow $2$-subgroup of $G$ is
of the form $P\cong R_{32,44}$ where $R_{32,44}$ has $2$-rank $2$ (see
Table~\ref{tbl:2rk}). 
A direct computation with MAGMA \cite{MAGMA} shows that $P$ has two
conjugacy classes of Klein-four subgroups which fuse in $G$. Finally,
Proposition~\ref{prop:2rank4} and~\cite[Prop.~5.2.10]{GLS3} show
that $G$ has rank $4$ and no maximal Klein-four subgroup for
$G=2.\fA_{n}$, or $2^\pm.\fS_n$ with $8\leq n\leq16$, except $G=2^{+}.\fS_n$
with $n\in\{8,9\}$. In this case, $G$ has $2$-rank $4$ and one conjugacy
class of maximal Klein-four subgroups.

Therefore Theorem~\ref{thm:ranktf} shows that
$TF(G)=\langle[\Omega_{G}(k)]\rangle\cong\IZ$ in all these groups except
$G=2^+.\fS_n$ with $n\in\{8,9\}$, in which case $TF(G)\cong\IZ^2$.

To complete the theorem, we need to find one more generator for $TF(G)$
in these latter two cases,
as we can pick $[\Omega_{G}(k)]$ by ``default''. This is an immediate
application of Theorem~\ref{thm52:CMT}. 
In our case the proof of Proposition~\ref{prop:2rank4} shows that a
noncentral involution $x$ of a maximal Klein-four subgroup $F$ of $G$
has centraliser whose Sylow $2$-subgroup has the form 
$\langle x\rangle\times Q_{16}$. 
So the integer $a$ in Theorem~\ref{thm52:CMT} is equal to $8$, while the
$\IZ$-rank of $TF(G)$ is $n_G=2$. So there is a
subquotient $M$ of the $kG$-module $\Omega_{G}^8(k)$ which is endotrivial
and subject to the conditions: 
$$ \Res^G_F(M)\cong\Omega_{F}^8(k)\oplus\proj\qbox{and}
\Res^G_E(M)\cong k\oplus\proj$$
for any non-maximal Klein-four subgroup $E$ of $G$. Then
Theorem~\ref{thm52:CMT} says that 
$TF(G)=\langle[\Omega_G(k)],[M]\rangle\cong\IZ^2$, as asserted.
\end{proof}
 
Proposition~\ref{prop:2-rank1} together with Proposition~\ref{prop:TTG2rank>2}
and Theorem~\ref{thm:tfchar2} complete the proof of Theorem~A, 
and the classification of endotrivial modules for the double covers of 
alternating and symmetric groups in characteristic $2$.

\bigskip
\section{Double covers in odd characteristic}\label{sec:charodd}

Throughout this section we let $p$ denote an odd prime and $k$ an
algebraically closed field of characteristic~$p$. The objective of this
section is to prove Theorem~B.
Building on \cite{CMN,CHM}, the question comes down to whether  there are
faithful endotrivial $kG$-modules. We start by showing that the answer
is negative in most cases, by adapting the proof of \cite[Thm. 4.5]{LMS} 
from simple modules to indecomposable modules in general.

\begin{thm}\label{thm:oddtorsion}
Let $p$ be an odd prime and let $G$ be a double cover of an
alternating or symmetric group. Then the inflation homomorphism
$\Inf_{G/Z(G)}^G : T(G/Z(G))\lra T(G)$ is an isomorphism in the
following cases. 
\begin{enumerate}[\hspace{6mm}\rm(1)]
  \item $G=2.\fA_{n}$ or $2^{\pm}.\fS_{n}$ with $n\geq p+4$; or
  \item $G=2.\fS_{6}$ and $p=3$.
\end{enumerate} \end{thm}

\begin{proof}
By Lemma~\ref{lem:resinf}(\ref{lem:resinf2}), the inflation homomorphism
$\Inf_{G/Z(G)}^G$ is injective, hence it suffices to show that no
faithful block of $kG$ contains an indecomposable endotrivial module. 
Assume that $M$ is an indecomposable endotrivial $kG$-module belonging
to a faithful block of $kG$.  
By Theorem~\ref{thm:lift}, $M$ is liftable to an endotrivial 
$\IC G$-module. 
Let $\chi_{M}$ denote the corresponding complex character. Then
$\chi_{M}$ is a sum of faithful characters in $\Irr(G)$. By 
Theorem~\ref{thm:lift}(\ref{thm:lift2}), it suffices to find a $p$-singular element 
$g\in G$ such that $|\chi_{M}(g)|\neq1$.  We use the notation for characters
and their parametrisation introduced in Section~\ref{sec:schur}.

First assume $G$ is one of $2.\fA_{n}$, or $2^{\pm}.\fS_{n}$ with $n\geq
p+4$ and let $g\in G$ be a $p$-singular element whose projection to
$G/Z(G)$ has cycle type $\mu=(p)(2)^{2}(1)^{n-p-4}$. Let
$\lambda\in\cD_{n}$. If $\lambda$ is odd, then for both characters
$\psi_{\lambda}^{\pm}\in\Irr(2^{\pm}.\fS_{n})$, we have
$\psi_{\lambda}^{\pm}(g)=0$ by \cite[Theorem 8.7(ii)]{HH}. 
A fortiori, their restrictions to $2.\fA_{n}$ must take value zero on
$g$ as well. Note that
$\psi_{\lambda}^{+}|_{2.\fA_{n}}=\psi_{\lambda}^{-}|_{2.\fA_{n}}\in\Irr(2.\fA_{n})$. 
If $\lambda$ is even, then for $\psi_{\lambda}\in\Irr(2^{\pm}.\fS_{n})$,
again by \cite[Theorem 8.7(iii) and (iv)]{HH}, we obtain
$\psi_{\lambda}(g)=0$. 
Moreover, the two irreducible constituents $\psi_{\lambda}^{\pm}$ of the
restriction of $\psi_{\lambda}$ to   $2.\fA_{n}$ are such that
$\psi_{\lambda}^{+}(g)=\psi_{\lambda}^{-}(g)$,  so that in fact
$\psi_{\lambda}^{\pm}(g)=0$. Consequently,
any sum of faithful irreducible characters of $G$ takes value zero on
$g$ and it follows that $\chi_{M}(g)=0$. Whence $M$ cannot be
endotrivial by Theorem~\ref{thm:lift}(\ref{thm:lift2}).

Suppose now that $G=2.\fS_{6}$ and that $p=3$. 
In this case consider an element $g'\in G$ which projects to
an element of cycle type $\mu'=(3)(2)(1)$. Then again by
\cite[Theorem 8.7(i),(ii) and (iii)]{HH}, we have
$\psi_{\lambda}(g')=0$ if $\lambda\in\cD(6)$ is even. 
Otherwise, $\psi_{\la}^{\pm}(g')=0$ if $\la=(6)$, and we calculate $\psi_{\mu'}^{\pm}(g')=
\pm i^{(\frac{6-3+1}2)}(\frac{3\cdot2\cdot1}2)^{\frac12}=\pm\sqrt{3}$. 
This forces $\chi_{M}(g')\in \sqrt{3}\IZ$ and thus $M$ cannot be
endotrivial by Theorem~\ref{thm:lift}(\ref{thm:lift2}).  
\end{proof}

The argument above does not apply to the group $G=2.\fA_{6}$ in
characteristic~$3$. However in this case $T(G)$ is already known since 
$G$ is in fact a finite group of Lie type in defining characteristic.

\begin{lem}\label{prop:char3A6S6}
In characteristic $p=3$, we have $T(2.\fA_{6})\cong \IZ/8\oplus\IZ$. 
Moreover the four faithful indecomposable torsion endotrivial
$k2.\fA_{6}$-modules are uniserial modules of dimension~$10$ with two
composition factors of dimension $2$ and one composition factor of
dimension~$6$.
\end{lem}

\begin{proof}
Let $G=2.\fA_6$ and $p=3$. Then $P\in\Syl_{3}(G)$ is elementary abelian of
order $9$ and  it is well known (\cite{ATLAS}) that  $G\cong\SL_2(9)$. Thus
\cite[Corollary~5.3]{CMN1} says that 
$T(G)\cong X(T)\oplus T(P)\cong\IZ/8\oplus\IZ$ where $P\in\Syl_3(G)$ and
$T$ denotes the torus of $\SL_2(9)$ of diagonal matrices with
determinant $1$ (see also \cite[\S 2.1.1]{BON}). 
Hence we have  $X(T)\cong \IF_9^\times \cong\IZ/8$. The assertion about the structure of 
the faithful endotrivial modules is
proven either by a direct computation with MAGMA \cite{MAGMA}, or
by inducing the linear characters of $kN_{G}(P)$ 
and using the decomposition matrix of $2.\fA_6$ (cf. \cite{CTblLib}). 
 
\end{proof}

\bigskip

Hence we are left with investigating endotrivial modules for the double
covers of alternating and symmetric groups of degree $p\leq n\leq p+3$
and when a Sylow $p$-subgroup is cyclic.

\begin{thm}\label{thm:odd-cyc}
Let $p$ be an odd prime and let $G$ be one of the double covers $2.\fA_{n}$
or $2^{\pm}.\fS_{n}$ with $p\leq n\leq min\{2p-1, p+3\}$. Let $e$ be the inertial
index of $B_{0}(kG)$. Then $T(G)\cong \IZ/2e\oplus\IZ/2$. 
\end{thm}

\begin{proof}
Let $P\in\Syl_{p}(G)$. Note that $P\cong C_{p}$. For convenience, 
set $N:=N_{G}(P)$, $X:=X(N)$ and $Z:=Z(G)$.  
By Theorem~\ref{thm:cyc}, we have $T(G)\cong T(N)$ and $T(N)$ is an
extension of $\IZ/2$ by $X$. 
Moreover, Theorem~\ref{thm:cyc}(\ref{thm:cyc3}) shows that $T(G)$ has a
subgroup $T_{0}(G)=\langle[\Omega_{G}(k)]\rangle\cong\IZ/2e$, which consists of the
indecomposable endotrivial modules in $B_{0}(kG)$.\par 

First assume that $G=2.\fA_{n}$. If $p>3$, we know from \cite[Proof of
  Prop. 4.6]{LMS}  that 
$$e=\left\{\begin{array}{ll}
\frac{p-1}2&\qbox{if $n\in\{p,p+1\}$, and}\\
p-1&\qbox{if $n\in\{p+2,p+3\}$,}
\end{array}\right.$$
and that in both cases $|X|/e=2$. The same holds if $p=3$ and
$n\in\{4,5\}$. So by Theorem~\ref{thm:cyc}(\ref{thm:cyc4}) there
are exactly two $kN$-blocks containing indecomposable endotrivial modules, namely the
principal block $B_{0}(kN)$ and a faithful block~$B_{1}$. 
Now it follows from Theorem~\ref{thm:cyc}(\ref{thm:cyc3}) and
\cite[Theorem 2.3]{Bess91a} that any indecomposable endotrivial $kN$-module
can be written as $M\otimes_{k}\theta$, where $M\in B_{0}(kN)$ and
$\theta\in X$ is a $B_{1}$-module. 
Consequently $T(N)$ has exponent $2e$ and we obtain $T(G)\cong
\IZ/2e\oplus\IZ/2$.\par

Next assume $G=2^{\pm}.S_{n}$\,. We need to determine whether $G$ has
faithful endotrivial modules.  
If $p>3$ and $n\in\{p,p+1\}$, then we know from \cite[Prop. 4.6(1),(2)]{LMS} that
$2.\fA_{n}$ has faithful simple endotrivial modules which are the
restrictions of faithful simple endotrivial $k2^{\pm}.\fS_{n}$-modules,
labelled by odd partitions. Thus as $G$ is a double cover of
$\fS_{n}$, we must have $|T(G)|=2|T(\fS_{n})|$. A direct computation shows that 
the same holds if $p=3$ and $n=4$. In addition,
$T(\fS_{n})\cong \IZ/2e$ by \cite[Theorem A(b)]{CMN} and the claim follows
by the same argument as in the case of $2.\fA_{n}$ showing that 
$T(N)$ has exponent  at most $2e$.
Therefore we conclude that 
$T(G)\cong \IZ/2e\oplus \IZ/2$. (Notice that in this case $e=p-1$ as
$N_{\fS_{n}}(P)\cong C_{p}\rtimes C_{p-1}$.)\par

If  $n\in\{p+2,p+3\}$, we claim that there cannot be any faithful indecomposable 
endotrivial modules. First if $p=3$ and $n=5$, it is easily checked
from  GAP  \cite{CTblLib} 
that there is no faithful block with inertial index $e$, so that the
claim follows from  Theorem~\ref{thm:cyc}(\ref{thm:cyc4}). 
Suppose now that $p>3$ and that there is some faithful block containing
an endotrivial $kG$-module. Then by
Theorems~\ref{thm:cyc}(\ref{thm:cyc4})
and~\ref{thm:lift}(\ref{thm:lift1}), we may assume that there is a
faithful simple endotrivial $kG$-module $S$ which lifts to a 
$\IC G$-module with complex character $\psi \in\Irr(G)$ labelled by a
partition $\la\in\cD(n)$ (see Section~\ref{sec:schur}).
Then $\Res^{G}_{2.\fA_{n}}(S)$ is endotrivial by
Lemma~\ref{lem:resinf}. 
If $\la$ is odd, then $\psi$ is of the form $\psi_\la^\pm\in\Irr(G)$ and 
$\Res^{G}_{2.\fA_{n}}(S)$ is simple endotrivial and affords
$\psi_{\lambda}^{+}|_{2.\fA_{n}}=\psi_{\lambda}^{-}|_{2.\fA_{n}}
\in\Irr(2.\fA_{n})$. 
This contradicts \cite[Prop. 4.6(3),(4)]{LMS}
since the faithful simple endotrivial $k2.\fA_{n}$-modules are
parametrised by even partitions. 
If $\la$ is even, then $\psi$ is of the form $\psi_{\la}$ and  since
$2.\fA_{n}\trianglelefteq G$ we must have
$\Res^{G}_{2.\fA_{n}}(S)=S_{1}\oplus S_{2}$, where $S_{1},S_{2}$ are
simple of the same dimension affording the two irreducible constituents
$\psi_{\la}^{\pm}\in \Irr(2.\fA_{n})$ of $\psi_{\la}$ by
\cite[Theorem~8.6]{HH}. So $S$ cannot be endotrivial because its
restriction to $2.\fA_n$ does not split as $M\oplus\proj$ for some
indecomposable endotrivial $k2.\fA_n$-module $M$.
Therefore 
$T(G)\cong \Inf_{G/Z}^{G}(T(G/Z))\cong T(\fS_{n})\cong
\IZ/2e\oplus\IZ/2$. (Notice that in this case $e=p-1$ as
$N_{\fS_{n}}(P)\cong (C_{p}\rtimes C_{p-1})\times \fS_{n-p}$.)  
\end{proof}


\section{The exceptional covering groups}\label{sec:exc}

In this final section we handle the exceptional Schur covers of the
alternating groups of degrees~$6$ and $7$. The only characteristics that
we need to discuss are $p\in\{2,3,5,7\}$. We proceed along the same lines
as for the double covers, separating the cases $p=2$ and $p\geq 3$.

From Table~\ref{tab:2rk} the Sylow $2$-subgroups of $3.\fA_6$ and $3.\fA_{7}$ are
dihedral of order $8$ and those of $6.\fA_6$ and $6.\fA_{7}$ are generalised quaternion
of order $16$. (Note that the $2$-local structure is the
same in degree $6$ and in degree $7$.)

\begin{prop}\label{prop:tt-exc}
Let $k$ be an algebraically closed field of characteristic $2$.
\begin{enumerate}[\hspace{6mm}\rm(1)]
\item Let $G=3.\fA_6$. Then 
$T(G)\cong\IZ/3\oplus\IZ^2$.
\item Let $G=3.\fA_7$. Then 
$T(G)\cong\IZ^2$.
\item Let $G=6.\fA_n$ with  $n\in\{6,7\}$.
Then $T(G)\cong  \IZ/2\oplus\IZ/4$.
\end{enumerate}
\end{prop}

\begin{proof}
First assume that  $G=6.\fA_n$ with $n\in\{6,7\}$. Then the arguments of the
proof of  Proposition~\ref{prop:2-rank1} apply to $G$ and give the
required result.

Suppose now that $G=3.\fA_n$ with $n\in\{6,7\}$. Let $P\in\Syl_{2}(G)$ and
put $N:=N_G(P)$. 
Then $N\cong C_3\times D_{8}$, and by Lemma~\ref{lem:resinf} the restriction induces an injective
group homomorphism $TT(G)\hookrightarrow X(N)\cong \IZ/3$. Thus, by
Lemma~\ref{lem:K(G)}, it suffices to determine whether the $kG$-Green
correspondents of the nontrivial modules in $X(N)$ are endotrivial. A
direct MAGMA computation shows the following.  
For $G=3.\fA_6$, the $kG$-Green correspondents of the two nontrivial
one-dimensional $kN$-modules are two faithful simple endotrivial
$k3.\fA_{6}$-modules of dimension~$9$ (note that it was proven in 
\cite[Theorem 4.9(5)]{LMS} that these modules are endotrivial,
 but not that they are torsion elements in $T(G)$).
Therefore, we have $TT(G)\cong \IZ/3$. 
For $G=3.\fA_7$, the $kG$-Green correspondents of the two nontrivial
one-dimensional $kN$-modules are indecomposable modules of dimension
$15$, which cannot be endotrivial because a trivial source endotrivial
$kG$-module must have dimension congruent to $1$ modulo~$8$, whence
$TT(G)\cong  \{[k]\}$. \par

Finally, for $TF(3.\fA_n)$ with $n\in\{6,7\}$, we recall that generators of
$TF(G)$ are known when $G$ has a dihedral Sylow $2$-subgroup. 
Namely they can be chosen to be $[\Omega_{G}(k)]$ and the class of one of
the two non-isomorphic indecomposable direct summands of the heart of
the projective cover of the trivial $kG$-module (see \cite[\S6]{AC}).
\end{proof}

\smallskip
Let us now turn to the endotrivial modules in odd characteristic.

\begin{prop}\label{prop:ex}
Let $k$ be an algebraically closed field of odd characteristic $p$. 
Let $G$ be one of the groups $3.\fA_{6}$, $6.\fA_{6}$, $3.\fA_{7}$ or $6.\fA_{7}$.
 \begin{enumerate}[\hspace{6mm}\rm(1)]
  \item If $p=3$, then 
$T(G)=\langle[\Omega_{G}(k)],[M]\rangle\cong\IZ^2$, where $M$ is a $kG$-module such that
$$ \Res^G_F(M)\cong\Omega_{F}^6(k)\oplus\proj\qbox{and}
\Res^G_E(M)\cong k\oplus\proj $$
for some maximal elementary abelian $3$-subgroup $F$ of $G$ of rank $2$
and for any elementary abelian $3$-subgroup $E$ of rank $2$ of $G$ not conjugate to $F$. 
  \item If $p\in\{5,7\}$, then  $T(G)$ is as given in Table \ref{tab:6A}.
  
  \begin{table}[htbp]
 \caption{Cyclic Sylow cases for  $6.\fA_6$ and $6.\fA_7$ }
  \label{tab:6A}
\[\begin{array}{|r|r|r|c|l|}
\hline
 G& p & X(N)& e& T(G)\\
\hline\hline
 3.\fA_6& 5  &  \IZ/3\oplus\IZ/2 &  2&       \IZ/3\oplus \IZ/4   \\
 6.\fA_6& 5  &  \IZ/3\oplus\IZ/4&  2&     \IZ/3\oplus\IZ/2\oplus \IZ/4   \\
\hline\hline
 3.\fA_7& 5 & \IZ/3\oplus\IZ/4&  4&     \IZ/3\oplus\IZ/8   \\
 6.\fA_7& 5 &  \IZ/3\oplus\IZ/8&  4&   \IZ/3\oplus\IZ/8\oplus \IZ/2    \\ \cline{2-5}
 3.\fA_7& 7 &   \IZ/3\oplus\IZ/3&  3&      \IZ/3\oplus\IZ/3\oplus\IZ/2   \\
 6.\fA_7& 7 &  \IZ/3\oplus\IZ/6&  3&   \IZ/3\oplus\IZ/6\oplus\IZ/2    \\
\hline
\end{array}\]
\end{table}  
\end{enumerate}  
\end{prop}

\begin{proof}
Assume that $p=3$. 
Then, $TT(G)$ is trivial since $G$ is a perfect group with a nontrivial
normal $3$-subgroup (see Lemma~\ref{lem:K(G)}).\par
Now, a Sylow $3$-subgroup $P$ of $G$ is extraspecial of order~$27$ and
exponent~$3$. We calculate with MAGMA (\cite{MAGMA}) that $P$ has two
  $G$-conjugacy classes of elementary abelian subgroups of rank $2$ and
that a representative of each of them is selfcentralising. 
Therefore Theorem~\ref{thm:ranktf} yields 
$TF(G)\cong  \IZ^{2}$. More precisely, let $F$ be an elementary abelian
subgroup of $P$ of rank $2$. 
The integer $a$ in Theorem~\ref{thm52:CMT} is equal to $6$ and there is a
subquotient $M$ of the $kG$-module $\Omega_{G}^6(k)$ which is endotrivial
and subject to the following conditions: 
$$ \Res^G_F(M)\cong\Omega_{F}^6(k)\oplus\proj\qbox{and}
\Res^G_E(M)\cong k\oplus\proj$$
for any elementary abelian subgroup $E$ of $G$ of rank $2$ not
conjugate to $F$. Then
Theorem~\ref{thm52:CMT} says that 
$TF(G)=\langle[\Omega_{G}(k)],[M]\rangle\cong\IZ^2$, as asserted.

Next assume that $p\in\{5,7\}$. Then a Sylow $p$-subgroup $P$ of $G$ is
cyclic of order $p$.  Let $e$ denote the inertial index of the principal
block $B_{0}(kG)$  (recall that $e=|N_{G}(P):C_{G}(P)|$). By
Theorem~\ref{thm:cyc},  $T(G)\cong T(N)$, where $N=N_{G}(P)$,  and
$T(N)$ is an extension of $X(N)$ by $T(P)\cong\IZ/2$, which splits when $e$ is
odd.  Now, both $X(N)$ and $e$ can be calculated directly for both characteristics. 
This is enough to work out the structure of $T(G)$  when $p=7$ since $e$ is odd in this case.
In addition for $p=5$ and $G\in\{3.\fA_{6},3.\fA_{7}\}$ we use the fact
that the module $\Omega_{G}(k)$ generates a cyclic subgroup of order~$2e$ of
$T(G)$, see Theorem~\ref{thm:cyc}(\ref{thm:cyc3}). 

 Finally for $p=5$ and $G\in\{6.\fA_{6},6.\fA_{7}\}$, by
 Theorem~\ref{thm:cyc}(\ref{thm:cyc4}) the number of blocks of $kN$ containing
 indecomposable endotrivial modules is $|X(N)|/e$, each of which contains precisely
 $2e$ of them. Furthermore, as in the
 proof of Theorem~\ref{thm:odd-cyc}, any indecomposable endotrivial module in a
 non-principal block of $kN$ can be obtained by tensoring  an
 indecomposable endotrivial module in the principal block $B_{0}(kN)$
 with a one-dimensional module in $X(N)$, which allows us to deduce the
 exponent of $T(N)$. The results of these computations are
 detailed in Table~\ref{tab:6A} and the claims follow.
\end{proof}


\bibliographystyle{alpha}
\bibliography{biblio.bib}


\end{document}